\documentclass[aap,MSNbibl,seceqn,dvips]{arximspdf}

\doi{10.1214/12-AAP864} 
\volume{23}
\issue{3}
\pubyear{2013}
\firstpage{1074}
\lastpage{1085}

\makeatletter
\newtheorem{theorem}{Theorem}[section]
\newtheorem{lem}{Lemma}[section]
\newtheorem{cor}{Corollary}[section]
\newtheorem{prop}{Proposition}[section]
\newtheorem{claim}{Claim}[section]

\newcommand{\ext}{\operatorname{ext}}
\newcommand{\sides}{\operatorname{sides}}

\makeatother

\begin{document}
\begin{frontmatter}

\title{Examples of nonpolygonal limit shapes in i.i.d. first-passage
percolation and infinite coexistence in spatial growth models}
\runtitle{Nonpolygonal limit shapes}

\begin{aug}
\author[A]{\fnms{Michael} \snm{Damron}\corref{}\thanksref{t1}\ead[label=e1]{mdamron@math.princeton.edu}}
\and
\author[A]{\fnms{Michael} \snm{Hochman}\thanksref{t2}\ead[label=e2]{hochman@math.princeton.edu}}
\runauthor{M. Damron and M. Hochman}
\affiliation{Princeton University}
\address[A]{Mathematics Department\\
Princeton University\\
Fine Hall, Washington Rd.\\
Princeton, New Jersey 08544\\
USA\\
\printead{e1}\\
\phantom{E-mail: }\printead*{e2}} 
\end{aug}

\thankstext{t1}{Supported by an NSF postdoctoral fellowship.}
\thankstext{t2}{Supported by NSF Grant 0901534.}

\received{\smonth{9} \syear{2010}}
\revised{\smonth{2} \syear{2011}}

%
\begin{abstract}
We construct an edge-weight distribution for i.i.d. first-passage
percolation on $\mathbb{Z}^{2}$ whose limit shape is not a polygon and
whose extreme points are arbitrarily dense in the boundary.
Consequently, the associated Richardson-type growth model can support
coexistence of a countably infinite number of distinct species, and the
graph of infection has infinitely many ends.
\end{abstract}

%

\end{frontmatter}

\section{Introduction}

Throughout this note $\mu$ denotes a Borel probability measure on
$[0,\infty)$ with finite mean and such that $\mu(\{0\})< p_c$, the
critical probability for bond percolation in $\mathbb{Z}^2$, and
$\mathcal{M}$ denotes the family of such measures. Let $\mathbb{E}$
denote the set of nearest-neighbor edges of the lattice
$\mathbb{Z}^{2}$, and let $\{\tau_{e}\dvtx e\in \mathbb{E}\}$ be
a\vspace*{1pt} family of i.i.d. random variables with marginal $\mu$
and joint distribution $\mathbb{P}=\mu^{\mathbb{E}}$. The
\textit{passage time }of a path $\gamma=(e_{1},\ldots,e_{n})\in
\mathbb{E}^{n}$ in the graph $(\mathbb{Z}^{2},\mathbb{E})$ is
$\tau(\gamma)=\sum _{i=1}^{n}\tau_{e_{i}}$, and for
$x,y\in\mathbb{Z}^{2}$ the \textit{passage time} from $x$ to $y$ is
\[
\tau(x,y)=\min_{\gamma}\tau(\gamma),
\]
where the minimum is over all paths $\gamma$ joining $x$ to $y$. A
minimizing path is called a \textit{geodesic} from $x$ to $y$.

The theory of first passage percolation (FPP)
is concerned with the large-scale geometry of the metric space
$(\mathbb
{Z}^{2},\tau)$. The following fundamental result concerns the
asymptotic geometry of balls. Write $B(t)=\{x\in\mathbb{Z}^{2}\dvtx
\tau
(0,x)\leq t\}$ for the ball of radius $t$ at the origin, and for
$S\subseteq\mathbb{R}^{2}$ and $a\geq0$, write $aS=\{ax\dvtx x\in S\}$.
%
\begin{theorem}[(Cox and Durrett~\cite{CoxDurrett})]
\label{thmshapethm}
For every $\mu \in\mathcal{M}$ there exists a deterministic,
compact,\setcounter{footnote}{2}\footnote{This is the only place where we use the assumptions
on $\mu$. If $\mu(\{0\})$ exceeds the critical percolation probability,
then $B_\mu=\mathbb{R}^2$ in an appropriate sense, and without finite
mean we could have $B_\mu =\{ 0\}$.} convex set $B_{\mu}$, with
nonempty interior, such that for every $\varepsilon>0$,
\[
\mathbb{P}\biggl((1-\varepsilon)B_{\mu}\subseteq\frac{1}{t}B(t)\subseteq
(1+\varepsilon)B_{\mu}\mbox{ for all large }t\biggr)=1.
\]
\end{theorem}

Little is known about the geometry of $B_{\mu}$, which is called the
\textit{limit shape}. It is conjectured to be strictly convex when $\mu$
is nonatomic, and nonpolygonal in all but the most degenerate cases,
but, in fact, there are currently no known examples of $\mu$ for which
these properties are verified (see~\cite{HM}). For a compact, convex
set $C\subseteq\mathbb{R}^2$ write
$\ext(C)$ for the set of extreme points and $\sides(C) = |{\ext(C)}|$,
so that $C$ is a polygon if and only
if $\sides(C)<\infty$. The best result to date, due to Marchand \cite
{Marchand},
is that under mild assumptions, $\sides(B_{\mu})\geq8$. Building on
results of Marchand, our purpose
of this note is to give the first examples of distributions for which
the limit shape is not a polygon. If $A$ and $B$ are subsets of
${\mathbb R}^2$ (with, say, the $\ell^1$-metric), we say that $A$ is
$\varepsilon$-\textit{dense} in $B$ if for each $x \in B$ there exists $y
\in A$ such that $\| x-y\|_1 < \varepsilon$.

\begin{theorem}
\label{thminfinitesides} For every $\varepsilon>0$ there exists $\mu
\in
\mathcal{M}$ $($with
atoms$)$ such that $B_{\mu}$ is not a polygon, that is, $\sides
(B_{\mu
})=\infty$, and $\ext(B_\mu)$ is $\varepsilon$-dense in $\partial
B_\mu
$. There exist nonatomic $\mu$ such that $\sides(B_\mu
)>1/\varepsilon$
and $\ext(B_\mu)$ is $\varepsilon$-dense in $\partial B_\mu$.
\end{theorem}

It is tempting to try to obtain a strictly convex limit shape by taking
a limit of measures $\mu_n$ such that $B_{\mu_n}$ have progressively
denser sets of extreme points, but unfortunately the limit one gets in
our example is the unit ball of $\ell^1$.

We also obtain examples of measures $\mu$ such that, at the points
$v\in
\ext(B_\mu)$ which lie on the boundary of the $\ell^1$-unit ball,
$\partial B_\mu$ is infinitely differentiable. This should be compared
with the work of Zhang~\cite{Z}, where such behavior was ruled out for
certain $\mu$. Last, as we explain in Section~\ref{secconstruction},
we can produce measures $\mu$ which are not purely atomic that have
$\sides(B_{\mu}) = \infty$.

Theorem~\ref{thminfinitesides} has implications for the Richardson
growth model, whose definition
we recall next. Fix $x_{1},\ldots,x_{k}\in\mathbb{Z}^{2}$ and imagine
that at time $0$ the site $x_{i}$ is inhabited by a species of type
$i$. Each species spreads at unit speed, taking time $\tau_{e}$
to cross an edge $e\in\mathbb{E}$. An uninhabited site is exclusively
and permanently colonized by the first species that reaches it, that is,
$y\in\mathbb{Z}^{2}$ is occupied at time $t$ by the $i$th species
if $\tau(y,x_{i})\leq t$ and $\tau(y,x_{i})<\tau(y,x_{j})$ for all
$j\neq i$. When there are \textit{unique geodesics}, that is, $\mathbb
{P}$-a.s. no two paths have the same passage time, as is the case when
$\mu$
is continuous, each site is eventually occupied by a unique species. We
shall also want to consider measures $\mu$ with atoms. The definition
of the model in this case is formally the same, but note that there may
be sites which are never colonized, that is, those sites $y$ for which
$\min_{1\leq i\leq k}\tau(y,x_i)$ is achieved by multiple $x_i$'s.

Given initial sites $x_{1},\ldots,x_{k}$, consider the set colonized
by the $i$th species,
\[
C_{i}=\{y\in\mathbb{Z}^{2}\dvtx y\mbox{ is eventually occupied by
}i\}.
\]
One says that $\mu$ \textit{admits coexistence of $k$ species} if for
some choice of
$x_{1},\ldots,x_{k}$,
\[
\mathbb{P}(|C_{i}|=\infty\mbox{ for all }i=1,\ldots,k)>0.
\]
Coexistence of infinitely many species is defined similarly. Notice
that if $\mu$ has atoms and a site $x$ is colonized by species $i$,
then the same will be true if we change the model by introducing an
arbitrary tie-breaking rule governing the infections of sites which are
reached simultaneously by more than one species. Thus, if coexistence
of $k$ species holds in our model, then the same is true under any
tie-breaking rule.

In the past ten years, there have been many studies related to
Richardson-type models, for instance, in questions related to the
asymptotic shape of infected regions~\cite{Gouere,Pimentel} and to
coexistence~\cite{HP,GM,Hoffman1,Hoffman}. Pertaining to the latter,
it is not known, even in simple examples, how many species can coexist.
When $\mu$ is the exponential distribution, H\"{a}ggstr\"{o}m and
Pemantle~\cite{HP} proved coexistence of 2 species (see~\cite{DH} for a
review of recent results on Richardson models, focused on exponential
passage times). Shortly thereafter, Garet and Marchand~\cite{GM} and
Hoffman~\cite{Hoffman1} independently extended these results to prove
coexistence of 2 species for a broad class of translation-invariant
measures on $(0,\infty)^\mathbb{E}$, including some non-i.i.d. ones.
Later, Hoffman~\cite{Hoffman} demonstrated coexistence of $8$ species
for a similarly broad class of measures by establishing a relation with
the number of
sides of the limit shape in the associated FPP. Using the same relation
we obtain the following theorem.
%
\begin{theorem}
\label{thmcoexistence}There exists $\mu\in\mathcal{M}$ $($with atoms$)$
which admits coexistence of infinitely many species. For each $k$ there
exist nonatomic $\mu\in\mathcal{M}$
admitting coexistence of $k$ species.
\end{theorem}

When $\mu$ is nonatomic Theorem~\ref{thmcoexistence} follows from
Theorem~\ref{thminfinitesides} and from Hoffman~\cite{Hoffman},
Theorem 1.4. For the atomic case we provide the necessary modifications
of Hoffman's arguments in Section~\ref{sechoffman-with-atoms}.

Finally, the \textit{graph of infection} $\Gamma(0)\subseteq\mathbb{E}$ is
the union over
$x\in\mathbb{Z}^d$ of the edges of geodesics from $0$ to $x$. This
terminology is consistent with the Richardson model when there are
unique passage times, in which case it is a tree, but note that in
general the graph of infection may also contain sites which were not
infected, that is, those where a tie condition exists, and in this way
we may obtain loops.
A graph has $m$ \textit{ends} if, after removing a finite set of vertices,
the induced graph contains at least $m$ infinite connected components,
and, if there are $m$ ends for every $m\in\mathbb{N}$, we say there are
infinitely many ends. Letting $K(\Gamma(0))$ be the number of ends in
$\Gamma(0)$, Newman~\cite{Newman} has conjectured for a broad class of
$\mu$ that $K(\Gamma(0))=\infty$. Hoffman~\cite{Hoffman} showed for
continuous distributions that in general $K(\Gamma(0)) \geq4$ almost surely.
%
\begin{theorem}
\label{thmends}There exist $\mu\in\mathcal{M}$ $($with atoms$)$ such
that $\mathbb{P}$-a.s.,\break $K(\Gamma(0)) = \infty$. For
each $k$ there exist nonatomic $\mu\in\mathcal{M}$ such that
$\mathbb
{P}$-a.s., $K(\Gamma(0)) \geq k$.
\end{theorem}

When there are unique geodesics, Hoffman's results imply that $K(\Gamma
(0))$ is at least $\sides(B_\mu)/2$ (\cite{Hoffman}, Theorem 1.4),
which proves Theorem~\ref{thmends} for nonatomic~$\mu$. In the case
that $\mu$ has atoms, Theorem~\ref{thmends} follows directly from
Theorem~\ref{thminfinitesides} (using the fact that the measure can be
made to be not purely atomic) and the following result, which we prove
in Section~\ref{secends}.
%
\begin{theorem} \label{thmends-atomic-case}
If $\mu$ is not purely atomic and $B_\mu$ has at least $s\in\mathbb{N}$
sides, then $K(\Gamma(0))$ is ${\mathbb P}$-a.s. at least
%
\begin{equation}\label{eqkdef}
k = 4\biggl\lfloor\frac{s-4}{12} \biggr\rfloor.
\end{equation}
\end{theorem}

See below Theorem~\ref{thmHoffman1} in Section~\ref{secends} for an
explanation of this bound.

\section{Background on the limit shape}

For any $x \in{\mathbb Z}^2$ let $m_\mu(x) =\break \lim_{n \to\infty}
\tau
(0, nx)/n$. This limit exists by Theorem~\ref{thmshapethm} and by a
theorem of Cox and Kesten, the map $\mu\mapsto m_\mu((1,0))$ is
continuous~\cite{CoxKesten}. By~\cite{Kesten}, Remark 6.18, this
continuity can be extended to other unit vectors $x$ and is actually
uniform over all of them. To describe this, endow $\mathcal{M}$ with
the topology
of weak convergence and for convenience fix a compatible
metric $d(\cdot,\cdot)$ on $\mathcal{M}$. Next, fix the $\ell^1$-metric
on $\mathbb{R}^2$, and write $A^{(\varepsilon)}$ for the $\varepsilon
$-neighborhood of $A\subseteq\mathbb{R}$. Let $\mathcal{C}$ denote
the space
of nonempty, closed, convex subsets of $\mathbb{R}^{2}$ endowed
with the Hausdorff metric $d_{H}$,
\[
d_{H}(A,B)=\inf\bigl\{\varepsilon\dvtx A\subseteq B^{(\varepsilon)}\mbox{
and }B\subseteq A^{(\varepsilon)}\bigr\}.
\]

\begin{theorem}[(Kesten)]\label{thmKesten} The map $\mu\mapsto B_{\mu}$
from $\mathcal{M}$ to $\mathcal{C}$ is continuous.
\end{theorem}

We shall use the following elementary semicontinuity property of the
map $A\mapsto\ext(A)$ for $A\in\mathcal{C}$.
%
\begin{lem}
Let $A\in\mathcal{C}$ and $x\in\ext A$. For every $\varepsilon>0$
there is a $\delta>0$ such that, if $A'\in\mathcal{C}$ and
$d_{H}(A,A')<\delta$,
then there exists $x'\in\ext A'$ with $\Vert x-x'\Vert
_{1}<\varepsilon$.\vadjust{\goodbreak}
\end{lem}
\begin{pf}
Choose a linear functional $T\dvtx\mathbb{R}^{2}\rightarrow\mathbb{R}$
and $\beta>0$ such that $T(x)>\beta$ and the set
$B=\{y\in A\dvtx T(y)\geq\beta\}$ has diameter less than $\varepsilon
/2$. Note that $x\in B$. Since $T$ is continuous, for small enough
$\delta$, if $d_{H}(A,A')<\delta$
then the set $B'=\{y\in A'\dvtx T(y)\geq\beta\}$ is nonempty and
satisfies $d_{H}(B,B')<\varepsilon/2$.
Since $T$ is linear, its maximum on the convex set $A'$ is attained
at some extreme point $x'\in\ext A'$, and by definition $x'\in B'$. Now
if $y$ is the closest point in $B$ to $x'$ then $\Vert x' - y \Vert_1
<\varepsilon/2$ and we also have $\Vert x - y\Vert_1 <\varepsilon/2$,
so $\Vert x-x'\Vert_{1}<\varepsilon$.\vspace*{-2.5pt}
\end{pf}

Combining this lemma with Theorem~\ref{thmKesten}, we have the
following corollary.\vspace*{-2.5pt}

\begin{cor}
\label{corcontinuity-of-ext-points}Let $\mu\in\mathcal{M}$. For
every $x_{1},\ldots,x_{k}\in\ext(B_{\mu})$ and $\varepsilon>0$ there
is a $\delta>0$ such that, if $\nu\in\mathcal{M}$ and $d(\nu,\mu
)<\delta$
then there are $y_{1},\ldots,y_{k}\in\ext(B_{\nu})$ such that
$\Vert
x_{i}-y_{i}\Vert<\varepsilon$
for $i=1,\ldots,k$.
\end{cor}

Next we recall some results about limit shapes for a special class of
measures. Given $0<p<1$, let $\mathcal{M}_{p}\subseteq\mathcal{M}$
denote the
set of measures $\mu\in\mathcal{M}$ with an atom of mass $p$ located
at $x=1$, that is, $\mu(\{1\})=p$, and no mass to the left of $1$,
that is, $\mu((-\infty,1))=0$. Limit shapes for $\mu$ of this form
were first studied by Durrett and Liggett~\cite{DurrettLiggett}.
Writing $\vec{p}_{c}$ for the critical parameter of oriented percolation
on $\mathbb{Z}^{2}$ (see Durrett~\cite{Durrett} for background),
it was shown that when $p>\vec{p}_{c}$ and $\mu\in\mathcal{M}_{p}$,
the limit shape $B_{\mu}$ contains a ``flat edge,'' or more
precisely, $\partial B_{\mu}$ has sides which lie on the boundary
of the $\ell^{1}$-unit ball. The nature of this edge was fully
characterized by Marchand
in~\cite{Marchand}. For $p\geq\vec{p}_{c}$, let $\alpha_{p}$ be
the asymptotic speed of super-critical oriented percolation on $\mathbb{Z}^{2}$
with parameter $p$ (see~\cite{Durrett}). Define points
$w_{p},w'_{p}\in
\mathbb{R}^{2}$
by
\begin{eqnarray*}
w_{p} & = & \bigl(1/2+\alpha_{p}/\sqrt{2}, 1/2-\alpha_{p}/\sqrt{2}\bigr),\\
w'_{p} & = & \bigl(1/2-\alpha_{p}/\sqrt{2}, 1/2+\alpha_{p}/\sqrt{2}\bigr).
\end{eqnarray*}
Let $[w_{p},w'_{p}]\subseteq\mathbb{R}^2$ denote the line segment with
endpoints $w_{p}$ and $w'_{p}$. It will be important to note
that $\alpha_{p}$ is strictly increasing in $p>\vec{p}_{c}$, so the
same is true of $[w_{p},w'_{p}]$.\vspace*{-2.5pt}
%
\begin{theorem}[(Marchand~\cite{Marchand})] Let $\mu\in\mathcal{M}_{p}$.
Then:
\begin{longlist}[(4)]
\item[(1)]$B_{\mu}\subseteq\{x\in\mathbb{R}^{2}\dvtx\Vert x\Vert
_{1}\leq
1\}$.
\item[(2)] If $p<\vec{p}_{c}$, then $B_{\mu}\subseteq\{x\in\mathbb
{R}^{2}\dvtx\Vert x\Vert_{1}<1\}$.
\item[(3)] If $p>\vec{p}_{c}$, then $\partial B_{\mu}\cap\{(x,y)\in
\mathbb
{R}^2\dvtx x+y=1\}=[w_{p},w'_{p}]$.
\item[(4)] If $p=\vec{p}_{c}$, then $\partial B_{\mu}\cap\{(x,y)\in
\mathbb
{R}^2\dvtx x+y=1\}=\{(1/2,1/2)\}$.\vspace*{-2.5pt}
\end{longlist}
\end{theorem}

As noted by Marchand, this implies $\sides(B_{\mu})\geq8$ for $\mu
\in
\mathcal{M}_{p}$
and $\vec{p}_{c}<p<1$, since $w_{p},w'_{p}$ and their reflections
about the axes are extreme points.\vspace*{-2.5pt}

\section{\texorpdfstring{Proof of Theorem \protect\ref{thminfinitesides}}{Proof of Theorem 1.2}}
\label{secconstruction}

Our aim is to construct a $\mu\in\mathcal{M}$ with\break
$\sides(B_{\mu})=\infty$. Fix any $p_{0}>\vec{p}_{c}$, $\mu_{0}\in
\mathcal{M}_{p_{0}}$
and a real parameter $\eta_{0}>0$. We will\vadjust{\goodbreak} inductively define a sequence
$p_{1}>p_{2}>\cdots>\vec{p}_{c}$, measures $\mu_{1}\in\mathcal{M}_{p_{1}}$,
$\mu_{2}\in\mathcal{M}_{p_{2}},\ldots,$ and $\eta_{1},\eta
_{2},\ldots>0$
such that for every $n\geq0$ and all $k\leq n$,
\begin{longlist}[(2)]
\item[(1)] If $\nu\in\mathcal{M}$ and $d(\nu,\mu_{k})<\eta_{k}$ then
$\sides
(B_{\nu})\geq k$, and
\item[(2)] $d(\mu_{k},\mu_{n})<\frac{1}{2}\eta_{k}$.
\end{longlist}
Note that, in particular, $\sides(B_{\mu_{k}})\geq k$ by (1). Assuming
$p_{k},\mu_{k}$ and $\eta_{k}$ are defined for $k\leq n$,
we define them for $n+1$. Fix $p_{n+1}\in(\vec{p}_{c},p_{n})$ and
set \mbox{$r=p_{n}-p_{n+1}>0$}. For $y>1$ construct $\mu_{n+1}^{y}$
from $\mu_{n}$ by moving an amount $r$ of mass from the atom at
$1$ to $y$, that is,
\[
\mu_{n+1}^{y}=\mu_{n}-r\delta_{1}+r\delta_{y}.
\]
We claim that for small enough $y>1$ and any sufficiently small choice
of $\eta_{n+1}>0$ (depending on the previous parameters), the measure
$\mu_{n+1}=\mu_{n+1}^{y}$ has the
desired properties. First, $\mu_{n+1}^{y}\rightarrow\mu_{n}$ weakly
as $y\downarrow1$, so, since $\mu_{n}$ satisfies (2), so does $\mu_{n+1}^{y}$
for all sufficiently small $y$.\vspace*{1pt}

Second,\vspace*{1pt} we claim that $\sides(B_{\mu_{n+1}^y})\geq n+1$ for $y$ close
enough to $1$. Indeed, since $r>0$ we have $w_{p_{n+1}}\neq w_{p_{n}}$.
Using (1), choose $n$ extreme points $x_{1},\ldots,x_{n} \in\ext
(B_{\mu_{n}})$
and let
\[
a=\min\{ \Vert x_{i}-x_{j}\Vert_{1},\Vert x_{i}-w_{p_{n+1}}\Vert
_{1}\dvtx i\neq j\}.
\]
Note that $a>0$ by Marchand's theorem. By Corollary \ref
{corcontinuity-of-ext-points}, for $y$ close enough to~$1$, for each
$i=1,\ldots,n$ we can choose an extreme point $x'_{i}$ of
$B_{\mu_{n+1}^{y}}$ with $\Vert x'_{i}-x_{i}\Vert<a/2$. By Marchand's
theorem, $B_{\mu_{n+1}^{y}}$ also\vspace*{1pt} has an extreme point at
$w_{p_{n+1}}$. By definition of $a$, these extreme points are distinct,
giving $\sides(B_{\mu_{n+1}^{y}})\geq n+1$.

Finally, by Corollary~\ref{corcontinuity-of-ext-points}, $\mu_{n+1}^{y}$
satisfies (1) for any sufficiently small choice of $\eta_{n+1}$.

Let $\mu$ be a weak limit of $\mu_{n}$. Then $d(\mu,\mu_{n})\leq
\frac
{1}{2}\eta_{n}$
for all $n$, so by (2), $\sides(B_{\mu})=\infty$. The proof is complete.

One can modify the construction in a number of ways in order to control
the resulting measure $\mu$. First, at each step, rather than creating
a new atom at $y$, one can instead add, for example, Lebesgue measure
on a small interval around $y$. In this way one can make the atom at
$1$ be the only atom of $\mu$.

Regarding the degree of denseness of the extreme points in the
boundary, note that at each stage, if $y$ is small enough and $p_{n+1}$
is close enough to $p_n$, the new extreme point we introduce can be
made arbitrarily close to $w_{p_n}$ (here we use that $\alpha_p$ is
continuous in $p>\vec p_c$, from~\cite{Durrett}), and in the limit we
can ensure an extreme point close to it. Thus, if we begin from $\mu
_0=\delta_1$ and choose $p_n$ so that $\lim p_n = \vec{p}_c$, and using
Marchand's result that the flat edge in $B_{\mu_n}$ then shrinks to a
point (and symmetry of the limit shape about the axes), we can ensure
$\varepsilon$-density of the extreme points of $B_\mu$.

For the second part of Theorem~\ref{thminfinitesides},
choose a sequence $\nu_{n}\in\mathcal{M}$ of continuous measures
converging weakly to $\mu$. By Corollary \ref
{corcontinuity-of-ext-points}, $\sides(B_{\nu_{n}})\rightarrow
\infty$
and if $\ext(B_\mu)$ is $\varepsilon$-dense in $\partial B_\mu$, then
the same holds for $B_{\nu_n}$ for sufficiently large $n$.

Regarding the remark after Theorem~\ref{thminfinitesides}, one may
verify that if at each stage $y$ is chosen close enough to $1$ and
$p=\lim p_n$, then $w_p$ is a $C^\infty$-point of $\partial B_\mu$.

\section{\texorpdfstring{Proof of Theorem \protect\ref{thmcoexistence}}{Proof of Theorem 1.3}}
\label{sechoffman-with-atoms}

Let us recall Hoffman's argument relating coexistence to the geometry
of the limit shape for continuous $\mu$ (\cite{Hoffman}, Theorem 1.6).
Extend $\tau$ to $\mathbb{R}^{2}\times\mathbb{R}^{2}$
by $\tau(x,y)=\tau(x',y')$ where $x'$ is the unique lattice point in
$x+[-1/2,1/2)^2$. Similarly, a geodesic between $x,y$ is a geodesic
between $x',y'$.
For $S\subseteq\mathbb{R}^{2}$, the Busemann function $B_{S}\dvtx\mathbb
{R}^{2}\times\mathbb{R}^{2}\rightarrow\mathbb{R}$
is defined by
\[
B_{S}(x,y)=\inf_{z\in S}\tau(x,z)-\inf_{w\in S}\tau(y,w).
\]
For $v\in\mathbb{R}^{2}$, write $S+v=\{s+v\dvtx s\in S\}$. If $v\in
\partial B_\mu$ is a point of differentiability and $w$ is a tangent
vector at $v$, let
$\pi_{v}$ denote the linear functional $av+bw\mapsto a$. Define
the lower density of a set $A\subseteq\mathbb{N}$ by $\underline{d}(A)
= \liminf\frac{1}{N}|A\cap\{1,\ldots,N\}|$. The following is a
rephrasing of~\cite{Hoffman}, Lemma 4.6.

\begin{theorem}[(Hoffman)] \label{thmbusemann} Let $\mu\in\mathcal{M}$
and let $v\in B_{\mu}$ be a point of differentiability of $\partial
B_{\mu}$
with tangent line $L\subseteq\mathbb{R}^{2}$. Then for every
$\varepsilon>0$
there exists an $M=M(v,\varepsilon)>0$ such that, if $x,y\in\mathbb{R}^{2}$
satisfy $\pi_{v}(x-y)>M$, then the set of $n$ such that
\[
\mathbb{P}\bigl(B_{L+nv}(y,x)>(1-\varepsilon)\pi_{v}(x-y)\bigr)>1-\varepsilon
\]
has lower density at least $1-\varepsilon$.
\end{theorem}

Hoffman's proof of this result does not use unique passage times.

Theorem~\ref{thmbusemann} is related to coexistence as follows. Suppose
$\sides(B_{\mu})\geq k$. We can then find $k$ points of differentiability
$v_{1},\ldots,v_{k}\in\partial B_{\mu}$ with distinct
tangent lines $L_{i}$, and in particular $\pi_{v_{i}}(v_{i}-v_{j})>0$
for all $j\neq i$. Fix $\varepsilon>0$ and choose $R>0$
large enough so that the points $x_{i}=Rv_{i}$ satisfy $\pi
_{v_{i}}(x_{i}-x_{j})>M(v_{i},\varepsilon/k^2)$.
Using the elementary relation $\underline{d}(\bigcap_{i=1}^{n}A_{i})\geq
1-\sum_{i=1}^{n}(1-\underline{d}(A_{i}))$,
for each $i$ we see that the set of $n$ such that
\[
\mathbb{P}\bigl(B_{L+nv_i}(x_{j},x_{i})>0\mbox{ for all }j\neq i\bigr)>1-\frac
{\varepsilon}{k}
\]
has lower density at least $1-\varepsilon/k$. Hence, with positive
probability (which can be made arbitrarily
close to $1$ by decreasing $\varepsilon$), for each $i$ there are
infinitely many $n$ such that $B_{L_{i}+nv_{i}}(x_{j},x_{i})>0$
for all $j\neq i$. For such an $n$, take $y_{i,n}\in L_{i}+nv_{i}$
to be the closest point (in the sense of passage times) to $x_{i}$;
by definition $y_{i,n}$ is reached first by
species $i$. The points $y_{i,n}$ are in $C_i$, so $|C_{i}|=\infty$ for
$i=1,\ldots,k$,
that is, coexistence occurs. Note that this argument does not use
unique passage times.


When $\sides(B_{\mu})=\infty$, one proves coexistence of infinitely
many types using the above arguments.
Choose a sequence $\{v_{i}\}_{i=1}^{\infty}\subseteq\partial B_{\mu}$
of points of differentiability of the boundary, ordered clockwise, say.
Given $\varepsilon>0$, define the points $x_{i}$ inductively by
$x_{i+1}=x_{i}+R_{i}(v_{i+1}-v_{i})$
for a sufficiently large $R_{i}>0$ so as to ensure that for $i\neq j$,
$\pi_{v_{i}}(x_{i}-x_{j})>M(v_{i},\varepsilon_{i,j})$, where
$\sum_{i,j}\varepsilon_{i,j}<\varepsilon$. Using the argument above, we
see that with probability $>1-\varepsilon$, any finite subcollection of
the $x_i$'s coexist. We now need to show that this implies that all
species coexist with probability $>1-\varepsilon$. Fix a configuration
in which every finite set of species coexist, and suppose that they do
not all coexist. Let $C_i$ denote the set of sites colonized by the
$i$th species when all species compete simultaneously. Since we are
assuming that coexistence does not occur, we have $|C_{i_0}|<\infty$
for some $i_0$. Let $\partial C_{i_0}$ denote the set of sites in
$\mathbb{Z}^2\setminus C_{i_0}$ which are adjacent to $C_{i_0}$, so
that this is a finite set. Every $u\in\partial C_{i_0}$ is colonized
by some species $j=j(u)$ at or before the time the species $i_0$
reaches $u$. It follows that when the species $i_0$ competes against
the finite set of species $\{j(u)\dvtx u\in\partial C_{i_0}\}$, it
still only colonizes the set $C_{i_0}$. This contradiction completes
the proof.

\section{\texorpdfstring{Proof of Theorem \protect\ref{thmends-atomic-case}}{Proof of Theorem 1.5}}\label
{secends}

Recall that $\Gamma(0)$ denotes the infection graph and that $K(\Gamma
(0))$ is the number of ends of $\Gamma(0)$. In the following discussion
we fix a measure $\mu\in\mathcal{M}$ and assume that $\mu$ is not
purely atomic. Consequently, there exists a Borel set $Q \subset
(0,\infty)$ such that $\mu(Q) > 0$ and $\mu(\{q\})=0$ for all $q \in Q$.
Note that $B_\mu$ has nonempty interior, is symmetric with respect to
reflections through the axes and is convex. Therefore, the origin is in
its interior.
%
\begin{prop}\label{propkends}
Suppose that ${\mathbb P}$-a.s. there exist $k$ infinite geodesics
$\gamma_1,\ldots,\gamma_k$ starting at $0$, edges $e_1,\ldots,e_k$ and
a finite set $V\subseteq\mathbb{Z}^2$, such that: \textup{(a)}~$e_i$ lies on
$\gamma_i$ but not on $\gamma_j$ for $i\neq j$, \textup{(b)} the endpoints
of $e_i$ are in $V$, \textup{(c)} $\tau_{e_i} \in Q$ and \textup{(d)} each pair
of geodesics is disjoint outside of $V$. Then ${\mathbb P}$-a.s.,
$K(\Gamma(0))\geq k$.
\end{prop}
\begin{pf}
Under our assumptions on $\mu$, with probability $1$ every pair of
edges with passage times in $Q$ has distinct passage times.
Consequently, geodesics between $x,y\in\mathbb{Z}^2$ can differ only in
edges $e$ with $\tau_e \notin Q$, and must share edges with $\tau_e
\in
Q$. We assume we are in this probability-$1$ event.

We claim that no two of the given geodesics are connected in $\Gamma(0)
\setminus V$. Suppose, for instance, that $\gamma_1,\gamma_2$ were
connected in $\Gamma(0) \setminus V$ by a path $\sigma$ which we may
assume is simple (not self-intersecting) and with endpoints $y_1 \in
\gamma_{1}$ and $y_2 \in\gamma_{2}$. Denote the sequence of vertices
in $\sigma$ by $y_1=v_{1},v_{2},\ldots,v_{m}=y_2$. Write $e=e_1$ and
let $J\subseteq\{1,\ldots,m\}$ denote the set of $j$ such that there
exists a geodesic $\sigma_{j}$ from $0$ to $v_{j}$ which contains\vadjust{\goodbreak} $e$.
We claim that $m \in J$. This leads to a contradiction because then
$\sigma_1$ and $\gamma_2$ are both geodesics connecting $0$ and $y_2$,
but only one of them, $\sigma_m$, contains $e$.

Clearly $1 \in J$. Suppose now that $j\in J$ with corresponding
geodesic $\sigma_{j}$. Write $f$ for the edge between $v_j$ and
$v_{j+1}$, and note that $f\neq e$ because the endpoints of $e$ are in
$V$ while those of $f$ are not. If $\tau(f)=0$, then we can adjoin $f$
to $\sigma_j$ to form $\sigma_{j+1}$. Suppose, therefore, that $\tau
(f)>0$, so that $\tau(0,v_j) \neq\tau(0,v_{j+1})$. If $\tau
(0,v_{j+1})>\tau(0,v_{j})$ then we adjoin $f$ to $\sigma_{j}$ and
obtain a geodesic $\sigma_{j+1}$ with
the desired properties. If $\tau(0,v_{j+1})<\tau(0,v_{j})$ and
$v_{j+1}$ lies on
$\sigma_{j}$, we remove $f$ from $\sigma_{j}$
to obtain $\sigma_{j+1}$. On the other hand, if $v_{j+1}$ does not
lie on $\sigma_{j}$, but $\tau(0,v_{j+1}) <\tau(0,v_j)$, then there is
a geodesic $\sigma'_{j+1}$ from $0$
to $v_{j}$ whose last edge is $f$. Because $\sigma'_{j+1}$ must reach
$v_{j}$ in the
same time as $\sigma_{j}$ does, it must pass through each of those
edges of $\sigma_{j}$
which have passage times in $Q$, and in particular through~$e$. We
remove $f$ from $\sigma'_{j+1}$ to obtain $\sigma_{j+1}$.
\end{pf}

Our goal is to establish the hypotheses of the proposition for $k$ as
in (\ref{eqkdef}). It is enough to show that there exist random
variables $m<M$ such that with probability one:
\begin{longlist}[(2)]
\item[(1)] There exist $k$ infinite geodesics $\gamma_1,\ldots,\gamma_k$
which are disjoint outside of $mB_\mu$.
\item[(2)] There are edges $e_i$ in $\gamma_i$, with endpoints in $MB_\mu
\setminus mB_\mu$, such that $\tau_{e_i} \in Q$.
\end{longlist}
This suffices because we can then set $V=MB_\mu$ in the proposition. To
show that such $m,M$ exist, it is enough to show that for every
$\varepsilon>0$ there exist deterministic integers $m<M$ such that each
of the conditions above holds on an event of probability
$>1-\varepsilon$.

Given $u,v,w\in\partial B_{\mu}$ which are points of differentiability
of $\partial B_{\mu}$, let $C(u,v,w)$ denote the open arc in $\partial
B_{\mu}$ from $u$ to $w$ containing $v$. We rely on the following
result, whose proof does not require unique geodesics. It is a
rephrasing of~\cite{Hoffman}, Lemma 4.7.
%
\begin{theorem}[(Hoffman)] \label{thmHoffman1} Let $u,v,w\in\partial
B_{\mu
}$ be points of differentiability of $\partial B_{\mu}$, let $L$ be the
tangent line at $v$ and write $C=C(u,v,w)$. Then for every $\varepsilon
>0$, there is an $M_{0}=M_{0}(\varepsilon)$ such that for every
$M>M_{0}$, the set of $n$ such that
\[
\mathbb{P}(\gamma\cap M\partial B_{\mu}\subseteq MC\mbox{ for all
geodesics }\gamma\mbox{ from }0\mbox{ to }L+nv) > 1-\varepsilon
\]
has lower density at least $1-\varepsilon$.
\end{theorem}

Henceforth, fix $k$ as in (\ref{eqkdef}) and $\varepsilon>0$ and for
$i=1,\ldots,k$ choose points $u_{i},v_{i},w_{i}\in\partial B_{\mu}$ and
lines $L_{i}$ as in the theorem, and such that the closed sets
$C_{i}=\overline{C(u_{i},v_{i},w_{i})}$ are pairwise\vspace*{1.5pt}
disjoint and do not intersect the boundary of the $\ell^1$ unit ball;
write $C=\bigcup _{i=1}^k C_i$. Note\vadjust{\goodbreak} that $k$ was picked so that such a
choice is possible: $\frac{1}{4}(\sides(B_\mu)-4)$ is the number of
distinct sides on each of the four curves in $\partial B_\mu$ which
constitute the complement of the $\ell^1$ unit ball; dividing this
number by 3 gives an upper bound on the number of triples we can choose
in each of these curves. Taking integer part and multiplying by $4$
gives $k$.
%
\begin{claim}\label{claimdensity}
There exists $M_0$ and $\rho>0$ such that with probability at least
$1-\varepsilon$, for all $M>M_0$, every $x\in MC$ and every geodesic
$\gamma$ from $0$ to $x$, at least $\rho M$ edges of $\gamma$ have
passage times in $Q$.
\end{claim}
\begin{pf}
Define edge weights $\{\tau'_{e}\dvtx e\in\mathbb{E}\}$ by the rule
that if $\tau_{e} \in Q$ then $\tau'_{e}=\tau_{e}+1$ and $\tau'_e =
\tau
_e$ otherwise. Let $\mu'$ denote the marginal distribution of~$\tau'_e$.

Choose $0<\eta<1$ so that $(1-\eta)C \cap B_{\mu'}=\varnothing$ (we can
do so by a theorem
of Marchand~\cite{Marchand}, Theorem 1.5, and the fact that $C$ is
disjoint from the $\ell^1$ unit ball). For a path $\sigma$ let
$N_Q(\sigma)$ be the number of edges of $\sigma$ with passage time in
$Q$. By Theorem~\ref{thmshapethm}, there is an event $A$ with
$\mathbb
{P}(A)>1-\varepsilon$
and an $M_{0}$ such that for all $M>M_{0}$ and $y\in M\partial B_\mu$,
the $\tau$-geodesic $\gamma$ from $0$ to $y$ satisfies $(1-\eta^2)M <
\tau(\gamma) < (1+\eta^2)M$, and similarly for $y'\in M\partial
B_{\mu
'}$ and the $\tau'$-length of $\tau'$-geodesics from $0$ to $y'$. We
claim that $A$ is the desired event. Indeed, let $M>M_{0}$ and let
$\gamma$ be a $\tau$-geodesic from $0$ to some $x\in MC$. We have
\[
\tau'(\gamma) = \tau(\gamma) + N_Q(\gamma) \leq(1+\eta^2)M +
N_Q(\gamma).
\]
On the other hand, $x=\frac{M}{s}y$ for some $y\in\partial B_{\mu'}$
and $s<1-\eta$, so
\[
\tau'(\gamma) \geq(1-\eta^{2})\frac{M}{1-\eta}.
\]
Combining these we find that $N_Q(\gamma) \geq(\eta-\eta^2)M$. We take
this to be $\rho M$.
\end{pf}

Let $\alpha>0$ denote the quantity
%
\begin{equation}\label{eqprojineq}
\alpha= \tfrac{1}{2} \min\{ \pi_{v_{i}}(x_i -x_j) \dvtx x_i\in C_i,
x_j\in C_j, i\neq j\}.
\end{equation}
Choose finite sets $D_i\subseteq C_i \cap\partial B_\mu$ with the
property that for every $i=1,\ldots,k$,
\[
C_i \subseteq\bigcup_{x\in D_i}\biggl(x+\frac{\alpha}{10}B_\mu\biggr).
\]
This property can be satisfied by compactness of $\bigcup_i C_i$ and the
fact that $B_\mu$ contains a neighborhood of the origin. Write
$D=\bigcup_{i=1}^k D_i$.

We can choose large integers $m$ and $M\gg m$ and a set $I\subseteq
\mathbb{N}$ of density $>1-\varepsilon$ such that, for $n\in I$, the
following statements hold with probability~$>1-\varepsilon$.

\renewcommand\thelonglist{(\Alph{longlist})}
\renewcommand\labellonglist{\thelonglist}
\begin{longlist}
\item\label{enbusemann-lower-bound} If $i\neq j$, then
$B_{L_{i}+nv_{i}}(x_{j},x_{i})\geq m\alpha$ for all $x_i\in mD_i$ and
$x_j \in mD_j$.

\item\label{enconcentrated-geos} Every geodesic $\gamma_{i,n}$ from
$0$ to $L_i + nv_i$ intersects $m \partial B_\mu$ in $mC_i$ and
intersects $M \partial B_\mu$ in $MC_i$.

\item\label{enlimitshapeineq1} $|\tau(0,x)-m|<\frac{m\alpha}{10}$ for
all $x \in mD$.

\item\label{enbdry-points} If $x\in mD$ and $y\in x+\frac{m\alpha
}{10}B_\mu$, then $\tau(y,x) < \frac{m\alpha}{5}$.

\item\label{enQ-weights} At least one edge on $\gamma_{i,n}\cap
(MB_\mu\setminus mB_\mu)$ has passage time in $Q$.
\end{longlist}

Indeed, for $m,M$ large enough the first two properties follow from
Theorems~\ref{thmHoffman1} and~\ref{thmbusemann}, and the third and
fourth from Theorem~\ref{thmshapethm} [for~\ref{enbdry-points} we
apply Theorem~\ref{thmshapethm} to each of the finitely many points in
$mD$ and intersect the events; note that the probabilities do not
depend on the point in question, only on $m$]. Last,
\ref{enQ-weights}~follows from the previous claim, since when $M$ is large
the number of edges on any geodesic from $0$ to $m\partial B_\mu$ is
smaller than $\rho M$.

Call $A_n$ the intersection of the above five events. Since $\mathbb
{P}(A_n)>1-\varepsilon$ for all $n\in I$, the event $A$ that $A_n$
occurs for infinitely many $n$ has $\mathbb{P}(A)>1-\varepsilon$. We
now consider only configurations in $A$. For each $n$ and $i$, fix
$\gamma_{i,n}$ as in~\ref{enconcentrated-geos}. We may choose a
(random) infinite set $I' \subseteq I$ such that for all $n\in I'$,
$A_n$ occurs and $J\subseteq I'$ such that $\lim_{n\in J} \gamma
_{i,n}\to\gamma_i$ for some infinite geodesics $\gamma_i$ originating
at $0$, that is, for every $r>0$ we have $\gamma_i \cap[-r,r]^2 =
\gamma_{i,n}\cap[-r,r]^2$ for all large enough $n\in J$. Henceforth,
we only consider such $n$.

Let $y_{i,n}$ be the first intersection point of $\gamma_{i,n}$ with
$mC_i$, and choose $x_{i,n}\in mD_i$ such that $y_{i,n}\in x_{i,n} +
\frac{m\alpha}{10}B_\mu$.
Then by~\ref{enbdry-points} we have $\tau(x_{i,n},y_{i,n})\leq
\frac
{m\alpha}{5}$, so by~\ref{enbusemann-lower-bound},
%
\begin{equation}\label{encloseness-of-busemann-funcs}
|B_{L_{i}+nv_{i}}(y_{j,n},y_{i,n})-B_{L_{i}+nv_{i}}(x_{j,n},x_{i,n})|<\frac
{2m\alpha}{5}\qquad
\mbox{for } i \neq j.
\end{equation}

\begin{claim}The $\gamma_i$'s are disjoint outside of $mB_\mu$.
\end{claim}
\begin{pf}
Suppose, for example, that $\gamma_1,\gamma_2$ intersect at some point
$z$ outside of $mB_\mu$. Then for large enough $n\in J$ the same is
true of $\gamma_{1,n}$ and $\gamma_{2,n}$. Then
\[
\tau(0,y_{1,n})+\tau(y_{1,n},z)=\tau(0,y_{2,n})+\tau(y_{2,n},z).
\]
By~\ref{enlimitshapeineq1} we have $|\tau(0,y_{1,n})-\tau
(0,y_{2,n})|<\frac{2m\alpha}{10}$, so
\[
|\tau(y_{1,n},z)-\tau(y_{2,n},z)|<\frac{2m\alpha}{10}.
\]
Write $\sigma_{1}$ for the part of $\gamma_{1,n}$ from $y_{1,n}$ to
$L_{1}+nv_{1}$. Let $\sigma_2$ be path which starts at $y_{2,n}$,
follows $\gamma_{2,n}$ until\vspace*{1pt} $z$ and then follows $\gamma
_{1,n}$ until $L_1+nv_1$. We find that
$|\tau(\sigma_{1})-\tau(\sigma_{2})|<\frac{2m\alpha}{10}$. But
$\gamma_{1,n}$ is a shortest path from $0$ to $L_{1}+nv_{1}$, so
$\sigma_{1}$ is a shortest path from $y_{1,n}$ to $L_{1}+nv_{1}$.
Hence, $B_{L_{1}+nv_{1}}(y_{2,n},y_{1,n})\leq\frac{2m\alpha}{10}$.
Combined with (\ref{encloseness-of-busemann-funcs}), this contradicts
\ref {enbusemann-lower-bound}.
\end{pf}

Finally, combining the last claim with~\ref{enQ-weights} establishes
the two claims stated after Proposition~\ref{propkends}. This
completes the proof of Theorem~\ref{thmends}.



\printaddresses

\end{document}